\documentclass[12pt,oneside]{amsart}
\usepackage[margin=1in]{geometry}% See geometry.pdf to learn the layout options. There are lots.

\usepackage{graphicx}
\usepackage{amssymb}
\usepackage{tikz}
\usepackage{tikz-cd}

\usepackage{amsmath,amsthm,amscd}
\usepackage{latexsym}
\usepackage[colorlinks,citecolor=red,pagebackref,hypertexnames=false]{hyperref}
\hypersetup{%
  colorlinks = true,
  citecolor=red,
  linkcolor  = black
}

\numberwithin{equation}{section}
\theoremstyle{plain}
\newtheorem{theorem}{Theorem}[section]
\newtheorem*{theorem*}{Theorem}
\newtheorem{lemma}[theorem]{Lemma}
\newtheorem{corollary}[theorem]{Corollary}

\theoremstyle{definition}

\theoremstyle{remark}

\newcommand{\Slim}[1]{\sum\limits_{#1}}

\title[Restriction Theory of the Parabola over $\mathbb{Z}/N\mathbb{Z}$ for Squarefree $N$]{On a Restriction Problem of Hickman and Wright for the Parabola over $\mathbb{Z}/N\mathbb{Z}$ for Squarefree $N$}
\author{Nathaniel Kingsbury-Neuschotz}
\date{\today}

\begin{document}
\begin{abstract}
Hickman and Wright proved an $L^2$ restriction estimate for the parabola $\Sigma$ over $\mathbb{Z}/N\mathbb{Z}$ of the form

  $$\left(\frac{1}{|\Sigma|}\sum\limits_{m\in\Sigma}|\widehat{f}(m)|^2 \right)^{\frac{1}{2}}\leq C_\epsilon N^\epsilon\cdot N^{-1}\left(\sum\limits_{x\in (\mathbb{Z}/N\mathbb{Z})^2}|f(x)|^\frac{6}{5}\right)^\frac{5}{6}$$

for all functions $f:(\mathbb{Z}/N\mathbb{Z})^2\rightarrow \mathbb{C}$ and any $\epsilon>0$, and showed that this bound is sharp when $N$ has a large square factor, especially for $N = p^2$ where $p$ is prime. In contrast, Mockenhaupt and Tao proved in the special case $N = p$ the stronger estimate

$$\left(\frac{1}{|\Sigma|}\sum\limits_{m\in\Sigma}|\widehat{f}(m)|^2 \right)^{\frac{1}{2}}\leq C N^{-1}\left(\sum\limits_{x\in (\mathbb{Z}/N\mathbb{Z})^2}|f(x)|^\frac{4}{3}\right)^\frac{3}{4}.$$

We extend the Mockenhaupt--Tao bound to the case of squarefree $N$, proving

$$\left(\frac{1}{|\Sigma|}\sum\limits_{m\in\Sigma}|\widehat{f}(m)|^2 \right)^{\frac{1}{2}}\leq C_\epsilon N^\epsilon\cdot N^{-1}\left(\sum\limits_{x\in (\mathbb{Z}/N\mathbb{Z})^2}|f(x)|^\frac{4}{3}\right)^\frac{3}{4},$$

and in fact a slightly sharper version with $C_\epsilon N^\epsilon$ replaced with $2^\frac{\omega(N)}{4}$, where $\omega(N)$ is the number of prime factors of $N$. We also discuss applications of this result to uncertainty principles and signal recovery.

\end{abstract}
\maketitle

\section{Introduction}
In \cite{HickmanWright}, Hickman and Wright discussed a number of restriction problems for Fourier transforms. One simple formulation asks, for a given set $\Sigma$ of frequencies described as the graph of a polynomial function, for what values of $s$ and $r$ is there a constant $C_\epsilon$ for each $\epsilon>0$ such that for any function $f: (\mathbb{Z}/N\mathbb{Z})^d\rightarrow \mathbb{C}$:

    \begin{equation} \label{eq:HickmanWrightRestriction} \left(\frac{1}{|\Sigma|}\sum\limits_{m\in\Sigma}|\widehat{f}(m)|^s \right)^{\frac{1}{s}}\leq C_\epsilon N^\epsilon\cdot N^{-d/2}\left(\sum\limits_{x\in (\mathbb{Z}/N\mathbb{Z})^d}|f(x)|^r\right)^\frac{1}{r}, \end{equation} 

    where the Fourier transform is normalized\footnote{Note that our normalization differs from that of Hickman and Wright, so that the right-hand side of the above inequality includes a factor of $N^{-d/2}$ not present in \cite{HickmanWright}.} by
    $$\widehat{f}(m) = N^{-d/2}\sum\limits_{x\in(\mathbb{Z}/N\mathbb{Z})^d}f(x)e^{-2\pi i x\cdot m/N}.$$

In the particular case of the parabola $\Sigma = \{(t, t^2):t\in\mathbb{Z}/N\mathbb{Z}\},$ if $s = 2$, the results of Hickman and Wright establish that such an estimate is available when $1\leq r\leq \frac{6}{5}$ (\cite{HickmanWright}, Theorem 6.5), and furthermore that for general $N$ this is sharp: for $r > \frac{6}{5},$ no such estimate may be made (see the discussion of the Knapp example on page 5 of \cite{HickmanWright}; the whole result described above is stated informally as Theorem 1.2).\footnote{Unlike most restriction problems in analysis, the Hickman--Wright restriction problem allows the factor of $C_\epsilon N^\epsilon$ on the right-hand side in order to account for factors arising from the number of divisors and/or the number of prime divisors of $N$; see Theorem \ref{th:main} and Corollaries \ref{NBoundedFactors} and \ref{GeneralN} to see how these issues arise quite concretely.}

For special classes of $N$, however, this result may be improved. For example, when $N = p$ is an odd prime, the restriction estimate holds for $1\leq r \leq \frac{4}{3}$; see for instance \cite{MockenhauptTao}. A restriction theorem with the same exponent over fields was also found for circles in \cite{IoKohCircles}, along with restriction theorems for general quadratic surfaces; for paraboloids, various authors have developed ever-stronger restriction estimates, see for instance \cite{IosevichKohLewko2020}, \cite{LewkoLewko2012}, \cite{Lewko2024}, \cite{KohPhamVinh2018}, and \cite{Koh2016}. In their sharpness example, Hickman and Wright require a large divisor $d|N$ such that $d^2|N$, so that the possible restriction estimates over $\mathbb{Z}/N\mathbb{Z}$ for squarefree $N$ were left unresolved. 

Another perspective on Fourier restriction comes from the related issue of Fourier decay: in \cite{MockenhauptTao}, the $r = \frac{4}{3}$ restriction estimate is deduced from an optimal square-root Fourier decay estimate of the form

$$|\widehat{\Sigma}(m)| \leq \frac{\sqrt{|\Sigma|}}{p} = p^{-\frac{1}{2}}$$

for all nontrivial characters $\chi_m$, which boils down to the evaluation of a quadratic Gauss sum. Over $\mathbb{Z}/N\mathbb{Z}$ for any composite $N$, even squarefree, the Fourier transform of the parabola does not have such optimal Fourier decay---this may be computed by hand within $(\mathbb{Z}/N\mathbb{Z})^2$, but also follows from the following general theorem (\cite{Kings1}, Theorems 2.27 and 3.4; see \cite{IoMPak} and \cite{Kings2} for similar results):

\begin{theorem*}
    Let $f(X_1,\dots, X_{d-1})$ be a polynomial in $\mathbb{Z}[X_1,\dots, X_{d-1}].$ Let $V_f(R)$ denote the solution set to $X_d = f(X_1,\dots, X_{d-1})$ over $R$. Suppose a sequence of finite rings $\{R_i\}_{i = 1}^\infty$ has the property that exponential sums over $V_f(R_i)$ have square root cancellation, in that there exists some constant $C>0$ such that for all $i$ and all nontrivial characters $\chi_m$ of $R_i^d$,

$$|\widehat{V_f(R_i)}(\chi_m)| \leq C\frac{\sqrt{|V_f(R_i)|}}{|R_i|^\frac{d}{2}} = C|R_i|^{-\frac{1}{2}}.$$

Then all but finitely many of the rings are fields or matrix rings of small dimension relative to $d$.\footnote{Precisely: if $d\geq 4$, all but finitely many of the rings are fields or $2\times 2$ matrix rings; if $d = 3$, all but finitely many of the rings are fields or $2\times 2$ or $3\times 3$ matrix rings; if $d = 2$, all but finitely many of the rings are fields or $2\times 2$, $3\times 3$, or $4\times 4$ matrix rings.} In the case where $f(X_1,\dots, X_{d-1}) = X_1^2 + \dots + X_{d-1}^2$, all but finitely many of the rings are fields.
\end{theorem*}

These considerations led Iosevich to ask the author whether the same is true for restriction exponents: if, for a sequence of finite rings of size tending to infinity, there exists some uniform constant $C$ such that the parabola satisfies an $s = 2, r = \frac{4}{3}$ restriction estimate, must the rings eventually be fields? In other words, can an optimal restriction estimate for the parabola hold over a ring that's not a field? Iosevich was particularly interested in the case of $\mathbb{Z}/N\mathbb{Z}$ for $N$ squarefree, and especially in the case $N = pq$: is the largest exponent for which such a uniform estimate holds $\frac{6}{5}$, $\frac{4}{3}$, or something intermediate between them? In this note, we answer this ``fields only'' question in the negative: in the case $N=pq$, and more generally when $N$ has a bounded number of prime factors, we obtain a restriction estimate with an exponent of $4/3$ and a uniform constant. For arbitrary squarefree $N$, we obtain the corresponding Hickman--Wright-style $N^\epsilon$-loss estimate for the same exponent $4/3$. In particular, we will prove the following theorem: 

\begin{theorem} \label{th:main}
    For any squarefree $N$, the parabola $\Sigma$ satisfies the following restriction estimate: for any $f:(\mathbb{Z}/N\mathbb{Z})^2\rightarrow\mathbb{C},$
    $$\left(\frac{1}{|\Sigma|}\sum\limits_{m\in\Sigma}|\widehat{f}(m)|^2 \right)^{\frac{1}{2}}\leq 2^{\frac{\omega(N)}{4}}\cdot N^{-1}\left(\sum\limits_{x\in (\mathbb{Z}/N\mathbb{Z})^2}|f(x)|^\frac{4}{3}\right)^\frac{3}{4},$$
where $\omega(N)$ denotes the number of (distinct) prime divisors of $N$.
\end{theorem}

Along the family $N = pq$ for $p$ and $q$ distinct primes, the constant does not grow with $N$. For general squarefree $N$, the worst-case growth of $\omega(N)$ occurs in the case where $N$ is a primorial, in which case $\omega(N)$ is on the order of 

$$\frac{\log(N)}{\log\log(N)}$$

by the Prime Number Theorem; this implies that for any $\epsilon > 0$, there is a constant $C_\epsilon$ such that for any squarefree $N$ and any $f:(\mathbb{Z}/N\mathbb{Z})^2\rightarrow\mathbb{C},$

$$\left(\frac{1}{|\Sigma|}\sum\limits_{m\in\Sigma}|\widehat{f}(m)|^2 \right)^{\frac{1}{2}}\leq C_\epsilon N^\epsilon \cdot N^{-1}\left(\sum\limits_{x\in (\mathbb{Z}/N\mathbb{Z})^2}|f(x)|^\frac{4}{3}\right)^\frac{3}{4},$$

as indeed there is an absolute constant $C$ such that for any squarefree $N$ and any $f:(\mathbb{Z}/N\mathbb{Z})^2\rightarrow\mathbb{C},$

$$\left(\frac{1}{|\Sigma|}\sum\limits_{m\in\Sigma}|\widehat{f}(m)|^2 \right)^{\frac{1}{2}}\leq CN^{\frac{1}{4\log\log(N)}}\cdot N^{-1}\left(\sum\limits_{x\in (\mathbb{Z}/N\mathbb{Z})^2}|f(x)|^\frac{4}{3}\right)^\frac{3}{4}.$$

\subsection{Applications to Uncertainty Principles and Signal Recovery}
As an application of the above estimates, we discuss some uncertainty principles and signal recovery results for the parabola, which illustrate the strength of our main result (readers not interested in these applications can skip to Section \ref{MainProofSect} with no loss of continuity). In \cite{IoMayeli}, Iosevich and Mayeli proved the following uncertainty principle. 

\begin{theorem}\label{th:iosevichmayelirestriction}
Suppose that $f: ({\mathbb Z}/N\mathbb Z)^d\to \mathbb C$ is supported in $E \subset ({\mathbb Z}/N\mathbb{Z})^d$, and $\widehat{f}:(\mathbb Z/N\mathbb{Z})^d\to \mathbb C$ is supported in $\Sigma \subset ({\mathbb Z}/N\mathbb{Z})^d$. Suppose that the restriction estimate 
    \begin{equation} \label{eq:mamarestriction} \left(\frac{1}{|\Sigma|}\sum\limits_{m\in\Sigma}|\widehat{f}(m)|^s \right)^{\frac{1}{s}}\leq C_{r,s} N^{-d/2}\left(\sum\limits_{x\in (\mathbb{Z}/N\mathbb{Z})^d}|f(x)|^r\right)^\frac{1}{r},\end{equation}
holds for $\Sigma$ for a pair $(r,s)$, $1\leq r\leq s$. If $1 \leq r \leq 2 \leq s$, then  
\begin{equation} \label{eq:iosevichmayeliUP} {|E|}^{\frac{2-r}{r}} \cdot |\Sigma| \ge \frac{N^d}{C^2_{r,s}}. \end{equation} 
\end{theorem} 

\vskip.125in 

Plugging in $d=2$, $\Sigma=\{(t,t^2): t \in {\mathbb Z}/N\mathbb{Z}\}$, $s=2$, and $r=\frac{6}{5}$ (the Hickman--Wright exponent\footnote{As this discussion serves as motivation, we are here suppressing the factors of $N^\epsilon$ and the dependence of our constant on $\epsilon$. As with our main result, this is literally true when $N$ has a bounded number of prime factors; for arbitrary $N$, the suppressed factors worsen the exponents by an arbitrarily small amount.}), we see that if $f: ({\mathbb Z}/N\mathbb{Z})^2 \to {\mathbb C}$ is a signal supported in $E$, with $\widehat{f}$ supported in $\Sigma$, then 
$$ {|E|}^{\frac{2}{3}} \ge \frac{N}{C^2_{6/5,2}},$$ or, equivalently, 
\begin{equation}\label{eq:HWUP} |E| \geq \frac{N^{\frac{3}{2}}}{C^3_{6/5,2}}. \end{equation} 
In contrast, Theorem \ref{th:main} implies a much stronger uncertainty principle when $N$ is squarefree. Indeed, plugging the conclusion of Theorem \ref{th:main} into Theorem \ref{th:iosevichmayelirestriction}\footnote{Using constant $C_{r, s} = C_\epsilon N^\epsilon$.} shows that if $f: ({\mathbb Z}/N\mathbb{Z})^2 \to {\mathbb C}$, $N$ squarefree, supported in $E$, with $\widehat{f}$ supported in the parabola, then for any $\epsilon>0$, there exists $C_{\epsilon}>0$ such that
$${|E|}^{\frac{1}{2}} \ge \frac{N^{1-2\epsilon}}{C^2_{\epsilon}},$$ 
or, equivalently, 
\begin{equation} \label{eq:improvedUP} |E| \ge \frac{N^{2-4\epsilon}}{C_{\epsilon}^4}.\end{equation} 

In the case when the number of prime divisors of $N$ is a fixed small number, the constant in (\ref{eq:improvedUP}) is absolute and does not grow with $N$, and we can take $\epsilon = 0$. 

\vskip.125in 

Iosevich and Mayeli (\cite{IoMayeli}) also showed that improved uncertainty principles can be used, via the method due to Matolcsi and Szucs (\cite{MS73}), and, independently, Donoho and Stark (\cite{DS89}), to obtain signal recovery results with less restrictive conditions. Suppose that $f: (\mathbb{Z}/N\mathbb{Z})^2 \to {\mathbb C}$, $N$ squarefree, supported in $E$, and the frequencies ${\{\widehat{f}(m)\}}_{m \in \Sigma}$ are unobserved, where $\Sigma$ is the parabola. If $f$ cannot be recovered uniquely and exactly, then there exists $g: (\mathbb{Z}/N\mathbb{Z})^2 \to {\mathbb C}$ supported in $F$ such that $|F|=|E|$, and $\widehat{f}(m)=\widehat{g}(m)$ for $m \notin \Sigma$. Let $h=f-g$. Then $|\operatorname{supp}(h)| \leq 2|E|$, and $|\operatorname{supp}(\widehat{h})| \leq N$, the size of the parabola. By (\ref{eq:improvedUP}), 
$$ |E| \ge \frac{N^{2-4\epsilon}}{2C^4_{\epsilon}}.$$

It follows that if we assume that 
$$ |E|<\frac{N^{2-4\epsilon}}{2C^4_{\epsilon}},$$ then we obtain a contradiction, which implies that $h \equiv 0$, and in particular exact and unique recovery. The method of recovery, as pointed out by Donoho and Stark, is suggested by the proof: we can take 
$$ \operatorname*{argmin}_{u: |\operatorname{supp}(u)|=|E|} {\|\widehat{f}-\widehat{u}\|}_2, $$ where the sum implicit in the $L^2$ norm is restricted to frequencies in $\Sigma^c$, i.e., the method of least squares. 

This algorithm is rather inefficient, and we are going to see that Theorem \ref{th:main}, properly interpreted, allows us to use a more efficient algorithm based on the celebrated Logan phenomenon (\cite{Logan}). We shall follow the approach of Burstein, Iosevich, Mayeli, and Nathan in \cite{BIMN2025}, which is a slight refinement of the approach of Iosevich, Kashin, Limonova, and Mayeli in \cite{IKLM24}. 

\vskip.125in 

The following result was established in \cite{DS89} using the celebrated Logan phenomenon. 

\begin{theorem} \label{theorem:DSLogan89Orig} (\cite{DS89}, Theorem 8) Let $f: ({\mathbb Z}/N\mathbb{Z})^d \to {\mathbb C}$ be supported in $E \subset ({\mathbb Z}/N\mathbb{Z})^d$. Suppose that $\widehat{f}$ is transmitted but the frequencies ${\{\widehat{f}(m)\}}_{m \in \Sigma}$ are unobserved, where $\Sigma \subset ({\mathbb Z}/N\mathbb{Z})^d$, with $|E| \cdot |\Sigma|<\frac{N^d}{2}$. Then $f$ can be recovered exactly and uniquely. Moreover,
\begin{equation} 
\label{equation: simpleL1recoveryOrig}
f = \operatorname*{argmin}_{g:\widehat{g}(m)=\widehat{f}(m)\text{ for } m \notin \Sigma} {\|g\|}_{L^1\left(({\mathbb Z}/N\mathbb{Z})^d\right)}.
\end{equation} 
\end{theorem} 

The algorithm implicit in this result is very efficient, with runtime roughly $O(N^2)$. If we apply this result directly to the parabola, the condition needed for the algorithm to yield the original signal $f$ is 
$$ |E|<\frac{N}{2}.$$

To obtain a less restrictive condition under which the algorithm (\ref{equation: simpleL1recoveryOrig}) yields $f$, we need to reformulate Theorem \ref{th:main} a bit. 

\begin{theorem} \label{th:maindual} Suppose $N$ is squarefree. Then given any $\epsilon>0$, there exists $C_{\epsilon}>0$ such that if $f: ({\mathbb Z}/N\mathbb{Z})^2\rightarrow\mathbb{C}$ has $\widehat{f}$ supported in the parabola, then 
\begin{equation} \label{eq:parabolabourgain} {\left( \frac{1}{N^2} \sum_{x \in {({\mathbb Z}/N\mathbb{Z})^2}} {|f(x)|}^4 \right)}^{\frac{1}{4}} \leq C_{\epsilon} N^{\epsilon} {\left( \frac{1}{N^2} \sum_{x \in ({\mathbb Z}/N\mathbb{Z})^2} {|f(x)|}^2 \right)}^{\frac{1}{2}}. 
\end{equation} 
\end{theorem} 

This result is dual to Theorem \ref{th:main}, and the proof is essentially just duality; for completeness we include it below. We shall need the following lemma from \cite{BIMN2025}, which follows from H\"older's inequality. 

\begin{lemma} \label{lemma:vershynintrick} (\cite{BIMN2025}, Lemma 1.22) Suppose that for $h: (\mathbb{Z}/N\mathbb{Z})^d \to {\mathbb C}$ with $\widehat{h}$ supported in $S \subset (\mathbb{Z}/N\mathbb{Z})^d$, 
\begin{equation} \label{eq:bourgainproxy} {\left( \frac{1}{N^d} \sum_{x \in (\mathbb{Z}/N\mathbb{Z})^d} {|h(x)|}^r \right)}^{\frac{1}{r}} \leq C(r) {\left( \frac{1}{N^d} \sum_{x \in (\mathbb{Z}/N\mathbb{Z})^d} {|h(x)|}^2 \right)}^{\frac{1}{2}} \end{equation} for some $r>2$. 

Then 
$$ {\left( \frac{1}{N^d} \sum_{x \in (\mathbb{Z}/N\mathbb{Z})^d} {|h(x)|}^2 \right)}^{\frac{1}{2}} \leq {(C(r))}^{\frac{r}{r-2}} \cdot \frac{1}{N^d} \sum_{x \in (\mathbb{Z}/N\mathbb{Z})^d} |h(x)|. $$
\end{lemma} 

It follows from Lemma \ref{lemma:vershynintrick} that under the assumptions of Theorem \ref{th:maindual}, 
\begin{equation} \label{eq:talagrandparabola} {\left( \frac{1}{N^2} \sum_{x \in (\mathbb{Z}/N\mathbb{Z})^2} {|f(x)|}^2 \right)}^{\frac{1}{2}} \leq {\left( C_{\epsilon} N^{\epsilon} \right)}^2 \cdot \frac{1}{N^2} \sum_{x \in (\mathbb{Z}/N\mathbb{Z})^2} |f(x)|. \end{equation} 

This leads us to the main result of this subsection. 

\begin{theorem} \label{th:loganparabola} Let $f: (\mathbb{Z}/N\mathbb{Z})^2 \to {\mathbb C}$, where $N$ is squarefree. Suppose that $f$ is supported in $E$ and the frequencies ${\{\widehat{f}(m)\}}_{m \in \Sigma}$ are unobserved, where $\Sigma$ is the parabola in $(\mathbb{Z}/N\mathbb{Z})^2$. Then for any $\epsilon>0$ there exists $C_{\epsilon}>0$ such that if 
$$ |E|<\frac{1}{4} \cdot N^{2-4 \epsilon} \cdot \frac{1}{C^4_{\epsilon}},$$ then $f$ can be recovered via Logan's algorithm 
$$ f = \operatorname*{argmin}_{u: \widehat{u}(m)=\widehat{f}(m) \ \text{\emph{for}} \ m \notin \Sigma} {\|u\|}_1.$$
\end{theorem}

As before, when $N$ has a bounded number of prime factors, the constant $C_\epsilon$ may be taken to be absolute and we may take $\epsilon = 0$.  

\section{Proof of Main Theorem}\label{MainProofSect}
The key tool in the proof of our main theorem is the following ``universal restriction estimate'' of Iosevich and Mayeli (\cite{IoMayeli}, Theorem 3.12), which we state under our normalization of the Fourier transform.

\begin{theorem}\label{th:IoMayeli} (\cite{IoMayeli}, Theorem 3.12)
    Let $\Sigma \subset (\mathbb{Z}/N\mathbb{Z})^d$ have the property that
    $$|\Sigma| = \Lambda_{\text{size}}N^{d/2}$$
    and
    $$|\{(x, y, x', y')\in U^4 : x+y = x' + y'\}| \leq \Lambda_{\text{energy}}\cdot |U|^2$$
    for every $U\subset \Sigma$.
    Then we have the following restriction estimate: for any $f:(\mathbb{Z}/N\mathbb{Z})^d\rightarrow\mathbb{C},$
    $$\left(\frac{1}{|\Sigma|}\sum\limits_{m\in\Sigma}|\widehat{f}(m)|^2 \right)^{\frac{1}{2}}\leq \Lambda_{\text{size}}^{-\frac{1}{2}}\cdot\Lambda_{\text{energy}}^{\frac{1}{4}}\cdot N^{-d/2}\left(\sum\limits_{x\in (\mathbb{Z}/N\mathbb{Z})^d}|f(x)|^\frac{4}{3}\right)^\frac{3}{4}.$$
\end{theorem}
We will apply the above theorem to the parabola $\Sigma = \{(t, t^2):t\in\mathbb{Z}/N\mathbb{Z}\}$, which clearly has $\Lambda_{\text{size}} = 1.$ For \textit{squarefree} $N$, we can bound the additive energy as follows:
\begin{lemma}
    Suppose $N$ is squarefree, and let $\Sigma$ denote the parabola in $(\mathbb{Z}/N\mathbb{Z})^2$ as above. Then for any $U\subset\Sigma$, we have:
    $$E(U) \leq 2^{\omega(N)}|U|^2$$
    where $E(U)$ denotes the additive energy of $U$ and $\omega(N)$ denotes the number of (distinct) prime divisors of $N$.
\end{lemma}
\begin{proof}
    The additive energy $E(U)$ is defined as $|\{(x, y, x', y')\in U^4 : x+y = x' + y'\}|.$ It suffices to show that given $x, y\in U$, there are at most $2^{\omega(N)}$ choices of pairs $x', y'\in \Sigma$ such that $x' + y' = x + y$. 
    
    Write $x = (t, t^2)$, $y = (s, s^2)$, $x' = (t', t'^2)$, and $y' = (s', s'^2).$ Let $k = (k_1, k_2) = x + y.$ If $x' + y' = x+y$, then 
    $$t' + s' = k_1$$
    and 
    $$t'^2 + s'^2 = k_2.$$
    It follows that 
    $$(t' - s')^2 = t'^2 -2t's' + s'^2 = 2k_2 - k_1^2.$$
    Any admissible pair $(t',s')$ determines a square root
    $z=t'-s'$ of $2k_2-k_1^2$. Conversely, for each fixed value of $z$, the system
    \[
    t'+s'=k_1,\qquad t'-s'=z
    \]
has at most one solution if $N$ is odd, and at most two solutions if $N$ is even, since any two solutions differ by a simultaneous shift by $N/2$. Thus it suffices to bound the number of possible values of $z$. In particular, the lemma follows if we establish that the equation
    $$z^2 \equiv C \pmod N$$
    has at most $2^{\omega(N)}$ solutions if $N$ is odd, and at most $2^{\omega(N) - 1}$ solutions if $N$ is even. But this follows from the Chinese Remainder Theorem---if $N = p_1p_2\cdots p_{\omega(N)},$ where the $p_i$ are distinct primes, solutions to 
    $$z^2 \equiv C \pmod N$$
    are in bijection to $\omega(N)$-tuples $(z_1, z_2,\dots, z_{\omega(N)})$ where
    $$z_i^2 \equiv C \pmod{p_i}.$$
    Each such congruence has at most two solutions modulo $p_i$ if $p_i$ is an odd prime and exactly one solution modulo 2, and there are $\omega(N)$ congruences, giving the desired bound. The lemma follows.
\end{proof}
Combining the above lemma with Theorem \ref{th:IoMayeli} gives us a restriction estimate, whose precise form depends on whether we assume $N$ has a bounded number of prime factors. For ease of reference, we state as corollaries the following two results, to which we alluded in the introduction:

\begin{corollary}\label{NBoundedFactors}
    For any squarefree $N$ having at most $K$ prime divisors, the parabola $\Sigma$ satisfies the following restriction estimate: for any $f:(\mathbb{Z}/N\mathbb{Z})^2\rightarrow\mathbb{C},$
    $$\left(\frac{1}{|\Sigma|}\sum\limits_{m\in\Sigma}|\widehat{f}(m)|^2 \right)^{\frac{1}{2}}\leq 2^{\frac{K}{4}}\cdot N^{-1}\left(\sum\limits_{x\in (\mathbb{Z}/N\mathbb{Z})^2}|f(x)|^\frac{4}{3}\right)^\frac{3}{4}.$$
    In particular, whenever $N = pq$ where $p$ and $q$ are distinct primes, for any $f:(\mathbb{Z}/N\mathbb{Z})^2\rightarrow\mathbb{C}$ we have that 
    $$\left(\frac{1}{|\Sigma|}\sum\limits_{m\in\Sigma}|\widehat{f}(m)|^2 \right)^{\frac{1}{2}}\leq \sqrt{2}\cdot N^{-1}\left(\sum\limits_{x\in (\mathbb{Z}/N\mathbb{Z})^2}|f(x)|^\frac{4}{3}\right)^\frac{3}{4}.$$
\end{corollary}
Finally, it follows from the Prime Number Theorem that the worst-case growth of $\omega(N)$ is on the order of
$$\frac{\log(N)}{\log\log(N)};$$
we therefore have:
\begin{corollary}\label{GeneralN}
    For any squarefree $N$, the parabola $\Sigma$ satisfies the following restriction estimate: for any $f:(\mathbb{Z}/N\mathbb{Z})^2\rightarrow\mathbb{C},$
    $$\left(\frac{1}{|\Sigma|}\sum\limits_{m\in\Sigma}|\widehat{f}(m)|^2 \right)^{\frac{1}{2}}\leq CN^{\frac{1}{4\log\log(N)}}\cdot N^{-1}\left(\sum\limits_{x\in (\mathbb{Z}/N\mathbb{Z})^2}|f(x)|^\frac{4}{3}\right)^\frac{3}{4}$$
    for some universal constant $C$ not depending on $N$.
    In particular, for any $\epsilon > 0$, there exists a $C_\epsilon>0$ such that for any squarefree $N$, the parabola $\Sigma$ satisfies the following restriction estimate: for any $f:(\mathbb{Z}/N\mathbb{Z})^2\rightarrow\mathbb{C},$
    $$\left(\frac{1}{|\Sigma|}\sum\limits_{m\in\Sigma}|\widehat{f}(m)|^2 \right)^{\frac{1}{2}}\leq C_\epsilon N^\epsilon\cdot N^{-1}\left(\sum\limits_{x\in (\mathbb{Z}/N\mathbb{Z})^2}|f(x)|^\frac{4}{3}\right)^\frac{3}{4}.$$
\end{corollary}

\vskip.25in 

\section{Proof of Theorem \ref{th:maindual}}
If $f \equiv 0$, the result is trivial. Else, we write

$$\sum\limits_{x}|f(x)|^4 = \sum\limits_{x}f(x)g(x),$$

where $g(x) = f(x)\overline{f(x)}^2$. Let $h = \overline{g}$. Now,

\[
\begin{split}
\sum\limits_{x}f(x)g(x) &= N^{-1}\sum\limits_{x}\sum\limits_{m}\widehat{f}(m)\Sigma(m)e^{2\pi im\cdot x/N}g(x)\\
&= \sum\limits_{m}\widehat{f}(m) \Sigma(m)\overline{\widehat{h}(m)}\\
&\leq \|f\|_2 \|\Sigma\widehat{h}\|_2 \\
\end{split}
\]

by Cauchy-Schwarz and Plancherel, where we have used $\Sigma$ to denote the indicator function of the parabola by abuse of notation. Applying the restriction estimate given by Theorem \ref{th:main} (or more precisely by Corollary \ref{GeneralN}), we have that for all $\epsilon > 0$,

$$\|\Sigma\widehat{h}\|_2 \leq C_\epsilon N^{-(1-\epsilon)} |\Sigma|^{\frac{1}{2}}\|h\|_{4/3},$$

so that

\[\begin{split}
\sum\limits_{x}|f(x)|^4 &\leq C_\epsilon N^{-(1-\epsilon)} |\Sigma|^{\frac{1}{2}}\|f\|_2\|h\|_{4/3}\\
&= C_\epsilon N^{-(1-\epsilon)} |\Sigma|^{\frac{1}{2}}\|f\|_2\|g\|_{4/3}\\
&= C_\epsilon N^{-(1-\epsilon)} |\Sigma|^{\frac{1}{2}}\|f\|_2\left(\Slim{x}\left(|f(x)|^3\right)^\frac{4}{3}\right)^\frac{3}{4}\\
&= C_\epsilon N^{-(1-\epsilon)} |\Sigma|^{\frac{1}{2}}\|f\|_2\left(\Slim{x}|f(x)|^4\right)^\frac{3}{4}.
\end{split}\]

It follows that

$$\|f\|_4 \leq C_\epsilon N^{-(1-\epsilon)} |\Sigma|^{\frac{1}{2}}\|f\|_2;$$

rearranging and using $|\Sigma|^\frac{1}{2} = N^\frac{1}{2}$ then gives the result.

\vskip.25in

\section{Proof of the signal recovery result (Theorem \ref{th:loganparabola})}

\vskip.125in 

Let $g$ be a minimizer as in Logan's algorithm, and choose $h$ so that $f=g+h$. We have  
$$ {\|g\|}_1={\|f-h\|}_1={\|f-h\|}_{L^1(E)}+{\|h\|}_{L^1(E^c)}$$
$$ \ge {\|f\|}_1+ \left( {\|h\|}_{L^1(E^c)}-{\|h\|}_{L^1(E)} \right)$$

since $f$ is supported in $E$. 

If we can show that 
$$ {\|h\|}_{L^1(E)}<{\|h\|}_{L^1(E^c)},$$ 

this will imply that ${\|g\|}_1>{\|f\|}_1$, which is impossible since $g$ is a minimizer. The resulting contradiction will show that $h \equiv 0$, thus completing the proof. 

We have 
\begin{equation} \label{eq:almost} {\|h\|}_{L^1(E)} \leq {|E|}^{\frac{1}{2}} \cdot {\|h\|}_{L^2(E)} \leq {|E|}^{\frac{1}{2}} \cdot {\|h\|}_2. \end{equation} 

By construction of $g$, we have that $\widehat{h}$ is supported on the parabola, and thus the estimate (\ref{eq:talagrandparabola}) implies that the right-hand side of (\ref{eq:almost}) is bounded by 
$$ {|E|}^{\frac{1}{2}} \cdot C^2_{\epsilon} N^{2 \epsilon} \cdot N^{-1} \cdot {\|h\|}_1. $$

It follows that if 
$$ {|E|}^{\frac{1}{2}} \cdot C^2_{\epsilon} N^{2 \epsilon} \cdot N^{-1}< \frac{1}{2}, $$ 

then 
$$ {\|h\|}_{L^1(E)}<{\|h\|}_{L^1(E^c)}.$$

This amounts to the condition 
$$ |E|<\frac{1}{4} \cdot N^{2-4 \epsilon} \cdot \frac{1}{C^4_{\epsilon}}, $$

as required.

\section{Acknowledgments}
The author would like to thank Alex Iosevich for posing this problem to him, and Alex Iosevich and Azita Mayeli for teaching him about restriction theory, uncertainty principles, and signal recovery. The author completed this work while partially supported by a Graduate Center Fellowship at the CUNY Graduate Center.

\end{document}